\documentclass[10pt]{article}
\font\smallit=cmti10

\usepackage{amssymb,latexsym,amsmath,epsfig,amsthm} 

\makeatletter

\renewcommand\section{\@startsection {section}{1}{\z@}
{-30pt \@plus -1ex \@minus -.2ex}
{2.3ex \@plus.2ex}
{\normalfont\normalsize\bfseries\boldmath}}

\renewcommand\subsection{\@startsection{subsection}{2}{\z@}
{-3.25ex\@plus -1ex \@minus -.2ex}
{1.5ex \@plus .2ex}
{\normalfont\normalsize\bfseries\boldmath}}

\renewcommand{\@seccntformat}[1]{\csname the#1\endcsname. }

\makeatother

\newtheorem{theorem}{Theorem}
\newtheorem{lemma}[theorem]{Lemma}

\newtheorem{corollary}[theorem]{Corollary}

\theoremstyle{definition}
\newtheorem{definition}{Definition}
\newtheorem{conjecture}{Conjecture}
\newtheorem{remark}{Remark}

\begin{document}

\begin{center}
\uppercase{\bf \boldmath A Structure Sheaf for Kirch Topology}
\vskip 20pt
{\bf Alexander Borisov}\\
{\smallit Department of Mathematics and Statistics, Binghamton University, Binghamton, New York, USA}\\
{\tt aborisov@binghamton.edu}\\ 
\end{center}
\vskip 20pt
\vskip 30pt

\centerline{\bf Abstract}
\noindent
Kirch topology on $\mathbb N$ goes back to 1969, and is remarkable for being Hausdorff,  connected, and locally connected. In this sense, it is analogous to the usual topology on $\mathbb C,$ yet, to the author's knowledge, there have been no Kirch topology analogs of the sheaf of complex-analytic functions until very recently. In our paper \cite{LIP} we constructed such natural sheaf of rings, the sheaf of locally LIP functions. In this paper we investigate some of its basic properties, primarily regarding zeroth and first cohomology and $\check{\textrm C}$ech cohomology with respect to covers by basic open sets.

\thispagestyle{empty}
\baselineskip=12.875pt
\vskip 30pt

\section{Introduction}

In this paper $\mathbb N$ will mean the set of all positive integers and $\mathbb N_0$ the set of all non-negative integers. By an arithmetic progression in $\mathbb N$ we mean an infinite arithmetic progression: a set of the form $a+d\mathbb N_0$ for some natural $a$ and $d.$

Kirch topology on $\mathbb N$ is defined as the topology with the basis of open sets formed by all infinite arithmetic progressions $a+d\mathbb N_0,$ where  $\gcd (a,d)=1$ and $d$ is square-free. This topology is Hausdorff, connected, and locally connected (cf. \cite{Kirch}, \cite{Szczuka}). 

There is also another version of this definition, where $\mathbb N$ is replaced by $\mathbb Z \setminus \{0\}.$ It has the same topological properties and can be easily generalized to other domains (cf. \cite{Clarketal}). In this paper we will present our results in the classical setting over $\mathbb N$, but all lemmas and theorems have natural analogs for $\mathbb Z \setminus \{0\}$ and proofs are valid with just minimal notational modifications. 

To explain our results, we first introduce some terminology.

\begin{definition} An arithmetic progression in $\mathbb N$ is called full if it is not contained in a strictly bigger arithmetic progression with the same common difference. Alternatively, full arithmetic progressions are intersections with $\mathbb N$ of doubly-infinite arithmetic progressions in $\mathbb Z.$
\end{definition}

\begin{definition}Suppose $X\subseteq \mathbb N$ is a set.  We will call the intersection of all full arithmetic progressions in $\mathbb N$ that contain $X$ the arithmetic convex closure (AC closure, for brevity) of $X,$ to be denoted by $<X>$.  If $X=\{x_1,x_2,\dots , x_n\},$ we will use $<x_1, x_2,\dots , x_n>.$
\end{definition}

\begin{definition} A subset $U$ of $\mathbb N$ will be called basic if it is a full arithmetic progression $a+d\mathbb N_0$ with $gcd(a,d) =1$ and square-free $d$.
\end{definition}

\begin{definition} A subset $U$ of $\mathbb N$ will be called almost basic if it is obtained from a basic set by removing a closed subset of (natural) density zero.
\end{definition}

\begin{remark} Because Kirch topology is Hausdorff, all finite sets are closed. So removing from an almost basic set a finite subset keeps it almost basic. In particular, all arithmetic progressions $a+d\mathbb N_0$ with $gcd(a,d) =1$ are almost basic.
\end{remark}

\begin{remark} If $|X|\leq 1 ,$ then $<X>= X.$ If $|X|\geq 2,$ then $<X>$ is itself a full arithmetic progression. If $X$ is an almost basic set, then $<X>$ is a basic set and $<X>\setminus \ X$ is a closed set of density zero.
\end{remark}

\begin{lemma} Suppose $U$ is an almost basic set. Then it is open and for any $n\geq 2$ and any finite subset $\{x_1,\dots ,x_n\}$ of $U$ the set $<x_1,\dots,x_n>\cap \ U$ is infinite.
\end{lemma}

\begin{proof} Openness is immediate. For the second property, $<x_1,\dots,x_n>$ is an arithmetic progression that is contained in $<U>$. Since $<U>\setminus\  U$ has density zero, the set   $<x_1,\dots,x_n>\cap \ U$ has positive density, so it is infinite.
\end{proof}

The following definitions and results were introduced and proved by the author in \cite{LIP}. We refer to \cite{LIP} for the proofs.

\begin{definition} Suppose $U$ is an infinite subset of $\mathbb Z.$ Then $f\colon U \to \mathbb Z$ is a {\it locally integer polynomial} ({\it LIP}, for short) {\it function} if the restriction of $f$ to any finite subset $X$ of $U$ can be given by a polynomial with integer coefficients.  For an infinite subset $U$ of $\mathbb Z,$ we denote by $\text{LIP}(U)$ the set of all  LIP functions on $U.$
\end{definition}

By the following easy lemma, the above condition is equivalent to the interpolation polynomial of $f$ on $X$ having integer coefficients.
 
\begin{lemma} [\cite{LIP}, Lemma 2] Suppose $X=\{x_0,\ldots ,x_d\}\subset \mathbb Z.$ Suppose $f_X(x)\in \mathbb Q[x]$ is a polynomial of degree at most $d,$ and $p(x)\in \mathbb Z[x]$ is such that   $p(x)=f_X(x)$ for all $x\in X.$ Then $f_X(x)\in \mathbb Z[x].$
\end{lemma}

For every non-empty Kirch-open $U$ the set $\text{LIP}(U)$ is naturally a ring, and the collection of all  $\text{LIP}(U)$ is a presheaf, a sub-presheaf of the sheaf of all integer-valued functions. Its sheafification gives the following important definition (cf. \cite{LIP}, Definition 6).
 
\begin{definition}
Suppose $U$ is a non-empty Kirch-open subset of $\mathbb N.$ Then a function $f:U\to \mathbb Z$ is called locally LIP if for every $a\in U$ there exists a Kirch-open  $U_a\subseteq U$ containing $a,$ such  $f_{|_{U_a}}\in \text{LIP}(U_a).$ We  denote the ring of locally LIP functions on $U$ by $\mathcal {LIP}(U).$ With the natural restriction maps, we get a sheaf of rings in Kirch topology, that we will call the arithmetic structure sheaf of $\mathbb N.$
\end{definition}

Clearly, every LIP function on $U$ is locally LIP, while the converse is, in general, false. The first goal of this paper is  to prove that for every almost basic set every locally LIP function is, in fact, LIP, as was conjectured in \cite{LIP}. In many respects, this current paper is a sequel to \cite{LIP}. While we will restate the relevant results, with precise references for proofs, it is recommended that the reader gets some familiarity with LIP functions from \cite{LIP} before reading this paper.

The paper is organized as follows. In section 2 we recall some more definitions and basic results from \cite{LIP}. In section 3 we prove that for almost basic sets every locally LIP function is LIP. We also give an example of a connected open set for which some locally LIP functions are not LIP. In section 4 we discuss the  sheaf cohomology  for $\mathcal{LIP}$,  primarily $H^1,$ as well as $\check{\textrm C}$ech cohomology with respect to covers by almost basic sets. In particular, we prove that  $H^1(U,\mathcal{LIP})=0$ for almost basic sets $U$. In section 5 we discuss some future directions and potential applications of our theory.

\section{Preliminary definitions and results}

As in \cite{LIP}, for every finite set $X$, on which $f$ is defined, we denote by $f_X=f_X(x)$ the interpolation polynomial for $f$ on $X,$ with the dummy variable $x$.

From \cite{LIP}, Theorem 1, if $X=\{x_1,x_2,\dots \},$ then every LIP function on $X$ can be uniquely written as 
$$a_0+a_1(x-x_1)+a_2(x-x_1)(x-x_2)+\dots$$
for some integers $a_0,a_1,a_2,\dots .$

\begin{definition} We will call the above formula the product-sum representation of a LIP function $f$ on $X.$
\end{definition}

\begin{remark} \label{Rpowerseries} We think of the product-sum as an arithmetic analog of a power series. This analogy is especially natural if $X=\{a,a+d,a+2d,\dots \}.$ Then the product-sum formula becomes $$\sum \limits_{k=0}^{\infty}a_k\prod \limits_{i=1}^k (x-a-(i-1)d).$$

If we define the ``$d$-derivative''$D: \textrm{LIP}(X) \to \textrm{LIP}(X) $ as $D(f) (x)=\frac{f(x+d)-f(x)}{d},$ then $D(\prod \limits_{i=1}^k (x-a-(i-1)d))=k\cdot \prod \limits_{i=1}^{k-1} (x-a-(i-1)d),$ similar to $\frac{d}{dx} x^k=kx^{k-1}.$
\end{remark}

\begin{remark} With the above in mind, a locally LIP function on a Kirch open  set $U$ is a function that is given in some neighborhood of every point as a convergent product-sum. This is analogous to complex-analytic functions being given locally as convergent power series.
\end{remark}

If $f:U\to \mathbb Z$ is not LIP, then there is some finite $X\subset U$ such that $f_X\notin \mathbb Z[x].$  Following \cite{LIP}, we call minimal such $X$ circuits.

\begin{definition}[\cite{LIP}, Definition 4.] Suppose $f\colon U\to \mathbb Z$ is a function {\bf not} in $\text{LIP}(U).$ A finite subset $X$ of $U$ is called a {\it circuit} for $f$ if $f_X\notin \mathbb Z[x] $ but $f_{X\setminus \{a\}}\in \mathbb Z[x]$ for every $a\in X.$ 
\end{definition}

We have the following easy but useful lemma.

\begin{lemma}[\cite{LIP}, Lemma 7] \label{Circuit Lemma} Suppose $X$ is a circuit for $f.$ Then

{\it (1)}  $f_X(x)=cx^{|X|-1}+\ldots$ with $c\notin \mathbb Z;$

{\it (2)} all elements of $X$ are congruent modulo $d,$ where $d$ is the denominator of $c$ in the reduced form. 
\end{lemma}

\begin{remark} The coefficient of $f_X$  for $x^{|X|-1}$ is known as the divided difference and is used a lot both in pure and in applied mathematics. The above lemma implies that $f\in \text{LIP}(U)$ iff all of its divided differences are in $\mathbb Z.$ But its most interesting and mysterious aspect is the second statement. Nontrivial arithmetic progressions appear to be at the heart of the theory of LIP functions, just like they are at the heart of Kirch topology. 
\end{remark}

The proof of the above lemma uses a formula that relates the interpolation polynomials for $f$ on two sets that are related by swapping a pair of elements. Following \cite{LIP}, we call it the Exchange Formula:

\begin{lemma}[\cite{LIP}, Lemma 6, Exchange Formula] Suppose $X\sqcup \{a,b\}$ is a finite subset of $\mathbb Z,$ and $f$ is a function on it. Then
$$f_{X\sqcup \{a\}}(x)=f_{X\sqcup \{b\}}(x)+c\cdot (a-b)\cdot \prod \limits_{z\in X} (x-z),$$
where $c$ is the coefficient for $x^{|X|+1}$ in $f_{X\sqcup \{a,b\}}(x).$
\end{lemma}

The next lemma is well known. It is included to justify the definition after it.

\begin{lemma} \label{Lconsistency} Suppose $U_1, \dots, U_n$ are infinite arithmetic progressions in $\mathbb N$ such that any two of them have non-empty intersection. Then $\bigcap _{i=1}^{n} U_i $ contains an infinite arithmetic progression. 
\end{lemma}

\begin{proof} 
Suppose $U_i=a_i+d_i\mathbb N_0,$ $D=\textrm{lcm} (d_1,\dots ,d_n),$ and $D=p_1^{k_1}\cdot \dots \cdot p_m^{k_m}$ is the prime factorization of $D.$ For every prime $p_j$ dividing $D$ consider the projections of $U_i$ to $\mathbb Z/ p_j^{k_j}\mathbb Z.$  Clearly, any two projections must intersect. From the classification of subgroups of $\mathbb Z/ p_j^{k_j}\mathbb Z$ and their cosets, for every two intersecting projections one must be contained in the other. Thus, for each $p_j$ the intersection of all projections is non-empty. By the Chinese Remainder Theorem, the set of all integers $a$ that are congruent to $a_i$ modulo $d_i$ is non-empty and contains arbitrarily large numbers. Take one such $a$. Then $a+D\mathbb N_0\subseteq \bigcap _{i=1}^{n} U_i .$ 
\end{proof}

\begin{definition} \label{Dstar} A set $U$ is called (almost) star-like if $U=U_1\cup \ldots \cup U_n,$ where all $U_i$ are (almost) basic sets and  $U_i\cap U_j\neq \emptyset$ for every $i$ and $j.$
\end{definition}

\begin{remark} Equivalently, almost star-like sets are obtained from star-like sets by removing density zero closed sets.
\end{remark}

\begin{remark} For an almost star-like set $U=U_1\cup \ldots \cup U_n,$ with $<U_i>=<a_i+d_i\mathbb N_0>$ by Lemma \ref{Lconsistency} we can assume that all $a_i$ are the same. We can also assume that no $d_i$ divides $d_j$.  With this condition, the basic sets $<U_i>$ are uniquely determined  by $U,$ as maximal arithmetic progressions that up to a set of density zero are contained in $U.$ Their intersection is an arithmetic progression with the common difference $D=\textrm{lcm} (d_1,\dots ,d_n),$ which is also a basic set.
\end{remark}

The following useful lemma is a direct consequence of several results of \cite{LIP}.

\begin{lemma} \label{Lstar} Suppose $U=U_1\cup \ldots \cup U_n$ where the $U_i$ are almost basic sets with pairwise non-trivial intersection. Suppose $f:U \to \mathbb Z$ is LIP on each $U_i$. Then $f$ is LIP on $U.$
\end{lemma}

\begin{proof} This is a simple generalization of Corollary 6 of \cite{LIP} and it follows from Theorem 9 and Lemma 9 of \cite{LIP} in the same way. Incidentally, non-emptiness of the pairwise intersections of $U_i$ implies that the intersection of all $U_i$ is not empty.
\end{proof}

We end this section with a simple lemma that helps to understand the nature of Kirch topology in relation to p-adic topology.

\begin{lemma}[Proximity Principle] \label{PP} Suppose $U$ is a Kirch-open set, $n\in \mathbb N,$ and $p$ is a prime. Suppose there is some $u\in U$ such that $u \equiv n \pmod p.$ Then for every $k\in \mathbb N$ there are infinitely many $u \in U$ such that $u \equiv n \pmod{p^k}.$
\end{lemma}

\begin{proof} Because $U$ is open, we have some $u+d\mathbb N_0\subseteq U,$ where $d$ is square-free. Suppose $n=u+p^i m,$ where $p$ does not divide $m.$ We need to find $x$ that solves the congruence $u + dx \equiv u+p^i m \pmod{p^k}.$ Since $i\geq 1$ and the order of $p$ in $d$ is at most $1$, this is always possible. And once we find one such $x,$ any $x+p^ky$ for integer $y$ will also work.
\end{proof}

\section{Locally LIP vs. LIP }

We start  this section with a lemma that will be later used to ``improve'' circuits.
 
\begin{lemma}[Circuit Improvement Lemma] \label{CircuitImprovement}Suppose $X=\{x_1,\dots,x_n\}$ is a circuit for $f$ and $2\leq m\leq n.$ Suppose $e \in  <x_1,\dots , x_m>$. Then for at least one $i\in \{1,\dots, m\}$ the restriction of $f$ to $X\cup \{e\} \setminus\{x_i\}$ is not in $\mathbb Z [x].$
\end{lemma}

\begin{proof} Suppose the highest degree coefficient for $f_X$ is $c.$ From Lemma \ref{Circuit Lemma}, $c\notin \mathbb Z.$ For every $i$ denote $f_i(x)=f_{X\setminus \{x_i\}} \in \mathbb Z[x].$  Then 
$$f_{\{e\}\cup X\setminus \{x_i\}}(x)=f_i(x)+c_i\prod \limits_{j\neq i} (x-x_j).$$

If for all $i$ this polynomial has integer coefficients, then all $c_i\in \mathbb Z.$ Consider arbitrary $i$ and $k$ from $1$ to $m.$ Plugging in $x=e$ into the above formulas for $i$ and $k$ and subtracting, we get 
$$0=f_i(e)+c_i\prod \limits_{j\neq i} (e-x_j)-(f_k(e)+c_k\prod \limits_{j\neq k} (e-x_j)),$$
$$0=(f_i(e)-f_k(e))+[c_i(e-x_k)-c_k(e-x_i)]\prod \limits_{j\neq i,k} (e-x_j).$$

On the other hand, from the Exchange Formula, we get
$$f_i(e)-f_k(e)=c(x_i-x_k)\prod \limits_{j\neq i,k} (e-x_j).$$
Plugging this into the previous formula and canceling $\prod \limits_{j\neq i,k} (e-x_j),$ we get
$$c(x_k-x_i)=c_i(e-x_k)-c_k(e-x_i).$$
Suppose $d$ is the difference of $<x_1,\dots,x_m>,$ which is the greatest common divisor of all $x_i-x_k.$  Then $e-x_i$ and $e-x_k$ are multiples of $d.$ Taking the appropriate integer linear combination of the equations above, we can get $(cd)$ on the left and some integer multiple of $d$ on the right. But this implies $c\in \mathbb Z,$ a contradiction.
\end{proof}

\begin{remark} \label{DivDiff3}The relation $c(x_k-x_i)=c_i(e-x_k)-c_k(e-x_i)$ is part of the general theory of divided differences. Its importance far exceeds its use in the above lemma. We will use it again later. 
\end{remark}

\begin{theorem} \label{Tmain+} Every locally LIP function on an almost basic set is LIP.
\end{theorem}

\begin{proof} Suppose $U=a+d\mathbb N_0\setminus S,$ where $\gcd (a,d)=1,$ $d$ is square-free, and $S$ has density zero. We will prove by induction on $n\geq 2$ that no locally LIP $f:U\to \mathbb Z$ can have a circuit of size $n.$ Because $f$ is locally LIP, for every $v\in U$ there is a neighborhood $U_v$ containing $v$ on which $f$ is LIP. We can take $U_v$ to be $v+d_v\mathbb N_0.$ (For $v=x_i,$ we shorten the notation to $d_i$ instead of $d_{x_i}$).

First, we prove the base: $n=2.$ Suppose $f$ is a counterexample, and $\{x_1,x_2\}$ with $x_1<x_2$ is a circuit for $f.$ We know that $x_1-x_2$ does not divide $f(x_1)-f(x_2).$ 
The arithmetic progression generated by $x_1$ and $x_2$ is proper and is, except for a set of density zero, contained in $U.$ Suppose $U_{x_1}=x_1+d_1\mathbb N_0,$ and  $U_{x_2}=x_2+d_2\mathbb N_0,$ where $d|d_1$ and $d|d_2.$ Denote by $P$ the product of all primes $p$ dividing $d_1d_2$ that do not divide $x_2-x_1.$ The intersection of the arithmetic progressions $0+P \mathbb N_0$ and $x_1+(x_2-x_1)\mathbb N_0$ is not empty, thus an arithmetic progression, thus has positive density. So we can find $e$ in it that lies in $U.$

Suppose $U_e=e+d_e\mathbb N_0.$ To prove that $U_e$ intersects with $U_{x_1}$ we need to show that $\gcd (d_1,d_e)$ divides $e-x_1.$ Suppose $p$ divides $d_1$ and $d_e.$ Since $\gcd(d_e,e)=1,$ $p$ does not divide $e.$ So $p$ divides $x_2-x_1,$ thus it divides $e-x_1.$ By Lemma \ref{Lstar}, the function $f$ is LIP on $U_{x_1}\cup U_e.$ Similarly, it is LIP on  $U_{x_2}\cup U_e.$ So $f(x_1)-f(e)$ is a multiple of $x_1-e,$ thus a multiple of $x_1-x_2.$ The same is true, replacing $x_1$ by $x_2,$ so, subtracting, we get that $f(x_1)-f(x_2)$ is a multiple of $x_1-x_2,$ a contradiction. 

To prove the step of the induction, suppose that $f:U \to \mathbb Z$ has no circuits of size less than $n$ and has a circuit of size $n$: $\{x_1,\dots, x_n\}.$ For each $x_i$ we have a basic set $U_{x_i}$ on which $f$ is LIP. If all pairs of these $U_{x_i}$ have nonempty intersections, then by Lemma \ref{Lstar} the function $f$ is LIP on their union, a contradiction. We will use Lemma \ref{CircuitImprovement} repeatedly to replace the original circuit with a circuit with the above property. Note that, since $f$ has no circuits of size less than $n,$ the set $X\cup \{e\} \setminus\{x_i\}$ in Lemma \ref{CircuitImprovement} is a circuit for $f.$

We proceed in $n-1$ steps. At each step we replace one of the existing $x_i$ (to be called the old  points) by a new one. At the beginning of $k$-th step we have $n+1-k$ old points and $k-1$ new ones. We find $e$ in the AC-closure of old points (which, up to a set of density zero, lies inside $U$), that additionally satisfies the divisibility condition below. Then we use Lemma \ref{CircuitImprovement} to replace one of the old points by this $e$. We will prove that after $n-1$ steps any two elements of the circuit have intersecting neighborhoods on which $f$ is LIP.

{\it Divisibility Condition.} We choose $e$ divisible by all primes that divide $some$ $x_i-x_j$ or $d_i$, for all points, old and new, except for those $p$ that divide $all$ pairwise differences of old points.

We say that two points with their corresponding neighborhoods are connected if these neighborhoods intersect. We will show that at every step of our replacement procedure all new points are connected to all other points, new or old. Clearly, it is enough to prove it for the newly chosen point $e$. For the point $e$ with the neighborhood $e+d_e\mathbb N_0$ to be connected to $x_i,$ it is sufficient (and necessary) that $e-x_i$ is divisible by $gcd(d_e,d_i)$. Recall that $gcd(e,d_e)=1$ and $e$ is by construction divisible by all primes that divide $d_i,$ except those primes $p$ that divide all differences between the old points at that step. So we only need to be concerned with these primes $p$. If $x_i$ is an old point, then, since $e$ is in the AC-closure of the old points, $p$ divides $e-x_i.$ If $x_i$ is a previously created new point, then $p|x_i$ by the divisibility condition for $x_i,$ unless all old points when $x_i$ was created  were congruent modulo $p.$ So $p$ does not divide $d_i$ or $p$ divides $e-x_i$.   
\end{proof}

\begin{remark} The above theorem implies that every almost basic set is connected, but this can be also proved much easier along the same lines as the connectedness of basic sets.
\end{remark}

\begin{theorem} \label{Infstar} Suppose $U$ is an open set and $f$ is a locally LIP function on $U.$ Suppose $a\in U.$ Denote by $U_a$ the union of all almost basic subsets of $U$ that contain $a.$ Then $f$ is LIP on $U_a.$ 
\end{theorem}

\begin{proof} Suppose $f$ is not LIP on $U_a.$ Then it has a circuit on $U_a,$ which is a finite set. So there is a star-like subset $V$ of $U_a$ on which $f$ is not LIP. But this is impossible by Theorem \ref{Tmain+} and Lemma \ref{Lstar}.
\end{proof}

Our next goal is to exhibit a connected Kirch open set on which not every locally LIP function is LIP. Specifically, we consider $X=\mathbb N \setminus 6\mathbb N.$ It is a union of three basic Kirch open sets: $1+3\mathbb N_0,$   $1+2\mathbb N_0,$ and  $2+3\mathbb N_0.$ 

We will denote the second set by $V, $ the complement of the second set in the first by $U$, and  the complement of the second set in the third by $W$. Note that $V$ is open, but $U$ and $W$ are not open. Note also that every locally LIP function on $X$ is LIP on $U\cup V$ and on $V \cup W.$ We will also order $U,$ $V,$ and $W$ in increasing order and denote by $U_n$ (resp. $V_n$ and $W_n$) the finite sets consisting of the first $n$ elements of $U$ (resp. $V$ and $W$).

The following definition and notation will help keep our formulas manageable.
\begin{definition} Suppose $A\subset \mathbb N$ is finite, and $x$ is an independent variable. Then the split of $A,$ denoted by $[A](x)$ (or simply $[A]$ if $x$ is assumed) is the monic split polynomial in $\mathbb Z[x]$ whose roots are elements of $A$. That is, $[A](x)= \prod \limits_{a\in A} (x-a).$  Naturally, $[\emptyset]=1.$
\end{definition}

Before proving our main result, we need some lemmas. The first lemma is not strictly necessary, but it elucidates the logic of the proofs of the subsequent ones.

\begin{lemma} Suppose $f(x)$ and $g(x)$ are in $\mathbb Z[x],$ and they take the same values on $V_m.$ Suppose $Q: U_i\cup V_m \cup W_j \to \mathbb Z$ takes the same values as $f$ on $U_i \cup V_m$ and as $g$ on $V_m\cup W_j.$ Suppose $k\in \mathbb N_0.$ Then  $2^k Q$ is LIP on $U_i\cup V_m \cup W_j$ if and only if $2^k(f(x)-g(x))$ belongs to the ideal generated by $[U_i][V_m]$ and $[V_m][W_j].$
\end{lemma}

\begin{proof}  If $2^k(f(x)-g(x))\in ([U_i][V_m],[V_m][W_j]),$ then there are $\mu(x)$ and $\tau(x)$ in $\mathbb Z[x]$ such that $2^k(f-g)=[V_m](\mu [U_i]-\tau [W_j]).$ Then $2^kQ$ is given on $U_i\cup V_m \cup W_j$ by an integer polynomial 
$$2^k f(x)-\mu(x)[U_i](x)[V_m](x) =  2^k  g(x)-\tau(x)[V_m](x)[W_j](x). $$
In the other direction, suppose $2^kQ$ is given on $U_i\cup V_m \cup W_j$ by an integer polynomial $R(x).$ Then $2^kf(x)-R(x)$ is zero on $U_i\cup V_m,$ so it equals $\mu(x)[U_i](x)[V_m](x)$ for some $\mu(x)\in \mathbb Z[x].$  Likewise, $2^kg(x)-R(x)=\tau(x)[V_m][W_j](x)(x)$ for some $\tau(x)\in \mathbb Z[x].$ Subtracting, we get $2^k(f-g)=[V_m](\mu [U_i]-\tau [W_j]),$ so $2^k(f(x)-g(x))\in ([U_i][V_m],[V_m][W_j]).$
\end{proof}

\begin{lemma} Suppose $f(x)$ and $g(x)$ are integer polynomials that take the same values on $V_m$. Denote $h(x)=f(x)-g(x) = p(x) \cdot [V_m](x). $ Suppose that $2^kp(x)$ belongs to the ideal $([U_i],[W_j])$ of $\mathbb Z [x].$ Then there exist polynomials $\tilde{f}$ and $\tilde{g}$  in $\mathbb Z[x]$ such that

1) $\tilde{f} \equiv f \pmod {[U_{i}][V_m]}$ (i.e., $\tilde{f}$ and $f$ take the same values on $U_i \cup V_m$);

2) $\tilde{g} \equiv g \pmod {[V_m][W_{j}]}$;

3) $\tilde{f} \equiv \tilde{g} \pmod {[V_{m+1}]}$;

4) if $\tilde{h}=  \tilde{f}- \tilde{g} =\tilde{p}\cdot [V_{m+1}],$ then $2^k\tilde{p}\in ([U_i],[W_j]).$  
\end{lemma}

\begin{proof}
Denote $c=v_{n+1}.$ We will construct $\tilde{f}$ and $\tilde{g}$ as $f+a[U_i][V_m]$ and $g+b[V_m][W_j]$ for some integer constants $a$ and $b$ to be determined later. Clearly, the first two conditions will be satisfied. The third one is equivalent to $\tilde{f}-\tilde{g}$ being a multiple of $[V_{m+1}](x)=[V_m](x) \cdot (x-c).$ This is, clearly, equivalent to the numerical identity $f(c)-g(c)+a[U_i](c)-b[W_j](c)=0.$

First, we show that the above equation on $a$ and $b$ is always solvable. Note that $2^kp(x)\in ([U_i](x),[W_j](x)),$ so $2^kp(c)\in ([U_i](c),[W_j](c)).$  So $2^kp(c)$ is a multiple of $d=\gcd([U_i](c),[W_j](c)).$ But $[U_i](c)$ and  $[W_j](c)$ are odd, from the definition of the sets $U,$ $V,$ and $W$. Therefore $p(c)$ is a multiple of $d,$ so $f(c)-g(c)=p(c)=a[U_i](c)-b[W_j](c)$ for some integers $a$ and $b$. Additionally, we note that the choice of the pair $(a,b)$ is not unique: given one such pair $(a,b)$ and any $e\in \mathbb Z,$ the pair $(a+e\frac{[W_j](c)}{d}, b+e\frac{[U_i](c)}{d})$ also works.

Now we want to satisfy the last condition. For every $(a,b)$ as above, $\tilde{h}(x)=\frac{p(x)+a[U_i](x)-b[W_j](x)}{x-c} \in \mathbb Z[x].$ Suppose $2^kp(x)=\mu(x)[U_i](x)-\tau(x)[W_j](x)$ for some $\mu(x)$ and $\tau(x)$ in $\mathbb Z[x].$ Then we want
$$2^k\tilde{h}(x)=\frac{(\mu(x)+2^k a)[U_i](x)-(\tau(x)+2^k b) [W_j](x)}{x-c} \in ([U_i](x),[W_j](x)).$$

Fix some $a$ and $b$. Because the numerator of the above fraction vanishes at $c$ by the choice of $(a,b),$ we have $(\mu(c)+2^k a) [U_i](c) = (\tau(c) +2^k b) [W_j ](c).$ So there exists $e'\in \mathbb Z$ such that $\mu(c)+2^k a = e'\frac{[W_j](c)}{d}$ and $\tau(c) +2^k b =e'\frac{[U_i](c)}{d}.$ Because $d$ is odd, we can choose $e\in \mathbb Z$ so that $d \mid (e'+2^ke).$ Then by changing $(a,b)$ to  $(a+e\frac{[W_j](c)}{d}, b+e\frac{[U_i](c)}{d})$ we get  $(\mu(c)+2^k a,\tau(c) +2^k b)$ to be an integer multiple of $([W_j](c), [U_i](c)).$ Then modulo $(x-c)([U_i](x),{W_j}(x))$ the numerator  of the expression for $2^k\tilde{h}(x)$ is equivalent to an integer multiple of $[W_j](c)[U_i](x) - [U_i](c)[W_j](x).$ Finally,

$$\frac{[W_j](c)[U_i](x) \!- \! [U_i](c)[W_j](x)}{x-c}\!=\!\frac{[W_j](c)\! -\![W_j](x)}{x-c}[U_i](x)\! -\! \frac{[U_i](c)\! -\! [U_i](x)}{x-c}[W_j](x),$$
so it belongs to $([U_i](x),[W_j](x)).$
\end{proof}

\begin{lemma} Suppose $f(x)$ and $g(x)$ are integer polynomials that take the same values on $V_m$. Denote $h(x)=f(x)-g(x) = p(x) \cdot [V_m](x). $ Suppose that $2^kp(x)$ belongs to the ideal $([U_i],[W_j])$ of $\mathbb Z [x].$ Then there exists a polynomial $\tilde{f}\in \mathbb Z[x]$ and $n\geq k$ such that
\end{lemma}

1) $\tilde{f} \equiv f \pmod {[U_{i}][V_m]}$;

2)  if $\tilde{h}=  \tilde{f}- g =\tilde{p}\cdot [V_{m}],$ then $2^n\tilde{p}\in ([U_{i+1}],[W_j]).$  

\begin{proof} 
Denote $c=v_{n+1}.$ We will construct $\tilde{f}$ as $f+a[U_i][V_m]$ for some integer constant $a.$ This will automatically satisfy the first condition.

To satisfy the second condition, suppose $2^kp(x)=\mu(x)[U_i](x)-\tau(x)[W_j](x).$ Note that $2^k\tilde{p}(x)=(\mu(x)+2^k a )[U_i](x)-\tau(x)[W_j](x).$  Suppose the resultant of $[U_{i+1}]=[U_i]\cdot (x-c)$ and $[W_j]$ is $2^{j}d,$ where $d$ is odd. We choose $a$ so that $(\mu(c)+2^k a)$ is a multiple of $d$. Then for $n\geq k+j$ we have  that $2^n\tilde{p}(x)$ is congruent modulo $[U_{i+1}]$ to $2^{n-k}(\mu(c)+2^k a)[U_i](x)-2^{n-k}\tau(x)[W_j](x).$ Because $2^{n-k}(\mu(c)+2^k a)$ is a multiple of  the resultant of $[U_{i+1}]$ and $[W_j],$ which belongs to the ideal $([U_{i+1}],[W_j]),$ we are done.
\end{proof}

\begin{remark} Switching the roles of $U$ and $W,$ we get a version of the above lemma for keeping $U_i$ and $f$ and replacing $W_j$ by $W_{j+1}$ and $g$ by $\tilde{g}$.
\end{remark}

\begin{theorem} \label{TnonLIP} There exists a locally LIP non-LIP function on $X=U\cup V \cup W.$
\end{theorem}
\begin{proof}
We start from a pair of polynomials $f_1(x)=0$ and $g_1(x)=x-1$ considered as functions on $U_1\cup V_1 =\{4,1\}$ and $V_1\cup W_1=\{1,2\}.$ They take the same values at $1\in V_1$ but putting them together to form a function on  $U_1\cup V_1\cup W_1$ gives a non-LIP function $\frac{1}{2}(x-1)(4-x).$ We will then use repeatedly Lemma 8, Lemma 9, and Remark 6  to grow this pair to a pair $(f_n,g_n)$ of integer polynomial functions on $U_n\cup V_n$ and $V_n\cup W_n,$ that take the same values on $V_n$. Putting them together, we will get a locally LIP function on $X$ that is not LIP, because it has a circuit $\{2,4\}.$
\end{proof}

\begin{remark} My student Mithun Veettil has proved that every locally LIP function on $X$ is $\mathbb Z[\frac{1}{2}]$-LIP, meaning that all of its interpolation polynomials on $X$ are in $\mathbb Z[\frac{1}{2}] [x]$ (cf. \cite{MithunThesis}).  It would be very interesting to figure out the exact structure of the quotient of $\mathcal{LIP}(X)$ by $\textrm{LIP} (X)$. For starters, this is a $\textrm{LIP} (X)$-module and an abelian pro-2-group. The last assertion has to do with the projective system of LIP functions on $X_n=U_n\cup V_n \cup W_n$ satisfying Mittag-Leffler property.
\end{remark}

Interestingly, if we take a union of the set $X$ above with, say, the basic set $1+5\mathbb N_0,$ then on the resulting set every locally LIP function  is LIP. This is part of the following theorem, that will be truly used in the next section.

\begin{theorem} \label{Nest} Suppose $U_1,\dots. U_n$ are almost basic sets. Suppose $V_1,\dots, V_m $ are pairwise intersecting almost basic sets that satisfy the following conditions:

1) every $U_i$ intersects with every $V_j$;

2) for every two different $x$ and $y$ in $U=\bigcup \limits_{i=1}^n U_i$ and every prime $p$ that divides $x-y$ there is some $v$ in $V=\bigcup \limits_{j=1}^m V_j$ such that $x\equiv v \pmod p.$

Then for $X=U\cup V$ we have $\mathcal{LIP}(X)=\textrm{LIP}(X).$

\end{theorem}
\begin{proof} Suppose $f$ is locally LIP on $U\cup V.$ Then by Theorem \ref{Tmain+} and Lemma \ref{Lstar} it is LIP on $U_i\cup V$ for every $i.$  Suppose $f$ is not LIP on $U\cup V.$ Then it has a circuit on $U\cup V;$ we choose such circuit $\{x_1,\dots x_k\}$ with the smallest number of elements not in $V$ among the circuits of the smallest size.  By the above,  that circuit must contain at least two elements in $U\setminus V,$ in different $U_i$. Suppose $p$ is a prime that divides the denominator of the divided ratio $c$ of that circuit (cf. Lemma \ref{Circuit Lemma}). Then all elements of the circuit are congruent modulo $p$. From the second condition, we can find $v\in V$ that is also congruent to them modulo $p$. By renumbering, we can assume that $x_1,\dots, x_l$ are in $U\setminus V,$ while $x_{l+1},\dots, x_k$ are in $V.$ Suppose $p^k$ is the largest power of $p$ that divides all differences of $x_1,\dots, x_l$. By the Proximity Principle (Lemma \ref{PP}), we can find $e$ in $V$,  not in our circuit, such that $u_1 \equiv \ldots \equiv u_l \equiv e \pmod {p^k}.$ We can assume that $p^{k+1}$ does not divide $x_1-x_2.$ As in the proof of Lemma \ref{CircuitImprovement} (cf. Remark \ref{DivDiff3}), we get $c(x_1-x_2)=c_2(e-x_1)-c_1(e-x_2),$ where $c_1$ and $c_2$ are divided differences obtained by replacing $x_1$ and $x_2$ respectively with $e$. By looking at this equation p-adically, and recalling that $p$ divides the denominator of $c$, we get that $p$ must divide the denominator of $c_1$ or $c_2.$ But then replacing $x_1$ or $x_2$ by $e$ gives a new circuit of the same size but with smaller number of elements outside $V$, which contradicts our choice of the circuit.
\end{proof}

\begin{definition} \label{Nest} We will call a set, with the distinguished collections of subsets $\{U_i\}$ and $\{V_j\}$ that satisfy the conditions of the above theorem a nest. Almost star-like sets are also considered to be nests, with $n=0$. We will say that a nest is made of the sets $\{U_i\}$ and $\{V_j\}$. We will call $\{V_j\}$ the core and $\{U_i\}$ the straw part of the nest.
\end{definition}

\section{$H^1$ and  Higher Cohomology}

Our intuition suggests that almost basic sets form a Leray cover for the arithmetic structure sheaf, that is $H^i(U,\mathcal {LIP})=0$ for all $i>0$ for almost basic sets $U$ (cf. \cite{Taylor}). From \cite{Taylor}, Theorem 7.8.5, this would imply that the sheaf cohomology groups $H^i(U,\mathcal {LIP})$ coincide with $\check{\textrm C}$ech cohomology using any cover by almost basic sets. We cannot prove it yet, but we will prove it in this section for $n=1$. We assume some familiarity with sheaf cohomology and $\check{\textrm C}$ech cohomology. Since we are primarily concerned with $n=1$, we will avoid the spectral sequences formalism.

\begin{remark}[on notation] When the desire for brevity of notation outweighs the risk of being insufficiently pedantic, we will say that a function $f: X \to \mathbb Z$ is LIP on $Y\subseteq X$ whenever the restriction of $f$ to $Y$ is in $\textrm{LIP}(Y).$ In the same spirit of suppressing the obviously necessary restriction maps, if $f_1: X_1 \to \mathbb Z$ and if $f_2: X_2 \to \mathbb Z$ are functions, we will say that $f_1-f_2$ is LIP on $X_1\cap X_2$ whenever the difference of the restrictions is LIP.
\end{remark}

In order to work with $H^1(U,\mathcal {LIP})$, it is convenient to make some new definitions.

\begin{definition} We denote by ${\mathcal Z}$ the sheaf of all  $\mathbb Z$-valued functions.  Note that $\mathcal{LIP}$ is a subsheaf of ${\mathcal Z}$. The presheaf $\textrm{ZL}$ is defined by $\textrm{ZL}(U)=\mathcal{Z}(U)/\mathcal{LIP}(U).$ Note that if $\mathcal{LIP}(U)=\textrm{LIP}(U),$ then $\textrm{ZL}(U)=\mathcal{Z}(U)/\textrm{LIP}(U).$  Finally, we define the sheaf $\mathcal{ZL}$ to be the sheafification of the presheaf $\textrm{ZL}.$ 
\end{definition}

\begin{lemma} \label{longexact} Suppose $U$ is a Kirch open set. Then

1) $\textrm{ZL}(U)$ naturally embeds into  $\mathcal{ZL}(U)$.

2) The quotient of $\mathcal{ZL}(U)$ by the $\textrm{ZL}(U)$ is isomorphic to $H^1(U,\mathcal {LIP})$.

3) For every $k\geq 2$ we have  $H^{k-1}(U,\mathcal {ZL}) \cong H^k(U,\mathcal {LIP})$.
\end{lemma}

\begin{proof}
Be definition, we have a short exact sequence of sheaves
$$0\to \mathcal{LIP}\to \mathcal{Z}\to \mathcal{ZL}\to 0.$$

Because $\mathcal{Z}$ is flasque, $H^i(U,\mathcal {Z})=0$  for all $U$ and all $i>0.$ By the long exact sequence of cohomology of sheaves, we get the following exact sequences:
$$0\to \mathcal{LIP}(U) \to \mathcal{Z} (U) \to \mathcal{ZL}(U) \to H^1(U,\mathcal {LIP}) \to 0; $$
$$0\to  H^{k-1}(U,\mathcal {ZL}) \to H^k(U,\mathcal {LIP}) \to 0, \ \ for\ all\  k \geq 2.$$
The result follows.
\end{proof}

As a corollary, $H^1(U,\mathcal {LIP})=0$ if and only if $\mathcal{ZL}(U) = \textrm{ZL}(U)$. By the definition of the quotient sheaf as the sheafification of the quotient presheaf, elements of $\mathcal{ZL}(U)$ can be represented as systems of integer-valued functions $f_i$ on some $U_i$ that cover $U$ such that for any intersecting $U_i$ and $U_j$ the difference of restrictions of $f_i$ and $f_j$ is locally LIP on $U_i\cap U_j.$ We can always choose these $U_i$ to be basic sets, so that locally LIP functions on $U_i\cap U_j$ are LIP. Such system represents an element in the presheaf $\textrm{ZL}(U)$ iff there is some $f: U\to \mathbb Z$ such that for each $U_i$ the difference $f-f_i$ is locally LIP (thus, LIP) on $U_i.$

\begin{remark}The above considerations basically reduce proving that $H^1 (U,\mathcal {Z})=0$  for an almost basic set $U$  to proving that  $\check{\textrm{C}}$ech cohomology $\check{H}^1_{C} (U,\mathcal {Z})=0,$ for a set of covers $\{C\}$ of $U$ by basic sets, that contains a refinement of every open cover of $U$.
\end{remark}

The following theorem is unexpectedly strong.

\begin{theorem} \label{MainH1} Suppose $U$ and $W$ are open sets, and $V=U\cap W.$ Then every LIP function on $V$ can be written as a difference of restrictions of a LIP function on $U$ and a LIP function on $W.$
\end{theorem}

\begin{proof}
We can assume $V \neq \emptyset ,$ so all three sets are infinite. Suppose $U=\{u_1,u_2,\dots \},$ $U_n=\{u_1,u_2,\dots, u_n\}$ and similarly for $V$ and $W$. Then we have the following lemma relating splits of these finite sets.

\begin{lemma} For every $n$ for all large enough $m$ $[V_m] \in ( [V_{m+1}], [U_n], [W_n]).$
\end{lemma}

\begin{proof} For large enough $m$  $V_m$ contains $U_n\cap V$ and $W_n\cap V.$ So $U_n\setminus V_m= U_n\setminus V$ and $W_n\setminus V_m= W_n\setminus V$. Suppose $R$ is the resultant of $[U_n\setminus V]$ and $[W_n\setminus V]$. Clearly $R\neq 0$ and $R\in [U_n\setminus V], [W_n \setminus V]).$ Therefore, $R[V_m] \in ([U_n],[W_n]).$  Recall that $[V_{m+1}](x)=[V_m](x) (x-v_{m+1}).$ So the ideal $[  ( [V_{m+1}], [U_n], [W_n]) \ : \ ([V_m] )]$ contains $R$ and $x-v_{m+1}$; we want to show that it contains $1$. To do this, it is sufficient to find for every prime $p$ dividing $R$ a polynomial $f(x)$ such that $f \cdot [V_m] \in ( [V_{m+1}], [U_n], [W_n])$ and $p$ does not divide $ f(v_{m+1})$.

For every $u_i\in U_n\setminus V$ there are two possibilities: there is $v\in V$ congruent to $u_i$ modulo $p$ or there is not. We will call those $i$ that such $v$ exists, of first kind and the others of second kind (relative to $p$). Because $V$ is open in Kirch topology, by the Proximity Principle (Lemma \ref{PP}) for $i$ of the first kind one can find $v_{i,p}$ that are congruent to $u_i$ modulo $p^{ord_p(R)}.$ Moreover, we can take all $v_{i,p}$ to be distinct and not in $U_n$. Fix this collection of $\{v_{i,p}\}$. If $m$ is large enough so that $V_m$ contains all  $v_{i,p},$ we take $f(x) = \prod \limits_{i\ of\ 2nd\  kind} (x-u_i). $

Note that $p$ does not divide $f(v_{m+1})$. Note also that $f(x)\cdot [V_m](x)$ is congruent modulo $p^{ord_p(R)}\mathbb Z[x]$ to a multiple of $[U_n](x).$ Replacing $f(x)$ by $\frac{R}{p^{ord_p(R)} }f(x),$ we get that $f(v_{m+1})$ is still not divisible by $p$ and $f(x)\cdot [V_m](x) \in ( [V_{m+1}], [U_n], [W_n]),$ as promised.
\end{proof}

Due to the above lemma, we can take a strictly increasing sequence  $\{m_n\}_{n=1}^{\infty}$ such that for each $n$ for all $m\geq m_n$ we have $[V_m]\in ([V_{m+1}],[U_n], [W_n])$. Note that by repeated application of this condition, we get that $[V_{m_n}]\in ([V_{m_{n+1}}],[U_n], [W_n]).$

Suppose $f\in \textrm{LIP} (V).$ Combining together the terms of  the product-sum for $f$, we can write it as $\sum_{n=0}^{\infty} f_n(x) [V_{m_n}](x)$ for some polynomials $f_n(x).$ 
Starting with $n=1$, we rewrite $[V_{m_n}]$ as a linear combination with polynomial coefficients of  $[U_n],$ $[W_n],$ and $[V_{m_{n+1}}]$. We do this repeatedly, for $n=1,2,\dots$ to rewrite the sum for $f(x)$ as a sum of product-sums with polynomial coefficients of $[U_n]$ and $[W_n].$ This splits the sum for $f(x)$ as sum (or difference) of LIP functions on $U$ and $W.$ Note that every element of $V$ is only affected by a finite part of this infinite procedure, so the formal equality of sums is a legitimate equality of functions on $V.$
\end{proof}

\begin{corollary} \label{GlueZL} Suppose $U$ and $W$ are open sets, and $V=U\cap W.$ Suppose $\mathcal{LIP}(V)=\textrm{LIP}(V).$ Suppose $\alpha \in \textrm{ZL}(U),$ $\beta\in \textrm{ZL}(W),$ and the restrictions of $\alpha$ and $\beta$ to $V$ are equal. Then there is $\gamma \in \textrm{ZL}(U\cup W)$ such that its restrictions to $U$ and $W$ are $\alpha$ and $\beta$ respectively.
\end{corollary}

\begin{proof} Take $a\in \mathcal{Z}(U)$ and $b\in \mathcal{Z}(W)$ that are representatives of $\alpha$ and $\beta$ respectively. Then the restrictions of $a$ and $b$ to $V$ differ by some $c\in \mathcal{LIP}(V)=\textrm{LIP}(V).$ By the theorem above, $c$ can be written as a difference of restrictions of LIP functions on $U$ and $W$, call them $a_1$ and $b_1$. Then $a-a_1$ and $b-b_1$ are equal on $V$, so they can be put together to get one function in $\mathcal{Z}(U\cup W)$, whose class $\gamma \in \textrm{ZL}(U\cup W)$ satisfies the desired conditions.
\end{proof}

\begin{remark} In the proof of the above theorem only openness of $V$ was used, while $U$ and $W$ can be arbitrary. This makes the result potentially even more significant.
\end{remark}

\begin{theorem} \label{H1nest}Suppose $X$ is a nest, made out of the straw part $\{U_i\}_{i=1}^n$ and the core $\{V_j\}_{j=1}^m$. Then the $\check{\textrm{Ce}}$ch cohomology $\check{H}^1_{C} (X,\mathcal {LIP})=0,$ where $C$ is the cover of $X$ by $U_i$ and $V_j$.
\end{theorem}

\begin{proof} Every element of  $\check{H}^1_{C} (X,\mathcal {LIP})$ is a collection of locally LIP (therefore, LIP) functions on the intersections of two of the sets that make up $X$, that satisfy the cocycle conditions on the triple intersections. Because $\check{H}^1_{C} (X,\mathcal {Z})=0,$  one can get a system of functions $F_i: U_i \to \mathbb Z$ and $G_j : V_j\to \mathbb Z$ such that the difference of their restrictions to the intersections of any two sets is equal to the original system. Equivalently, we have a collection of sections $f_i\in \textrm{ZL}(U_i)$ and $g_j\in \textrm{ZL}(V_j)$ that match on all intersections. To show that our element in $\check{H}^1_{C} (X,\mathcal {LIP})$ is zero, is equivalent to showing that there are some LIP functions $\tilde{F_i}\in \textrm{LIP}(U_i)$ and   $\tilde{G_j}\in \textrm{LIP}(V_j)$ such that  $F_i-\tilde{F_i}$ and $G_j-\tilde{G_j}$ match on all intersections. In other words, we need to show that the element $s\in \mathcal{ZL}(X),$ obtained by gluing together all $f_i$ and $g_j,$ actually belongs to $\textrm{ZL}(X)$.

 Define an increasing collection of sets $X_1,X_2,\ldots , X_{n+m}$ as follows. For $k\leq m$ we take $X_k=\bigcup_{j=1}^k V_j$. Then for $m+1\leq k \leq n+m,$ we take $X_k=V \cup (\bigcup_{i=1}^{k-m} U_i)$. Note that all $X_k$ are nests, and $X_{n+m}=X$. We are going to show, step by step, that the restriction of $s$ to $X_k$ belongs to $\textrm{ZL}(X_k)$ for $k=1,2,\ldots, n+m$. 
 
We first apply Corollary \ref{GlueZL} repeatedly for $k=1,\ldots, m-1$ to $X_k$ and $V_{k+1}$ to show that the restriction of $s$ to $V_m=V$ is in  $\textrm{ZL}(V)$; call it $g$. Next, for each $i$ from $1$ to $n$ the elements $f_i$ and $g$ form an element in $\textrm{ZL}(U_i\cup V)$. For $k\geq m$,  $X_{k} \cap (U_{k-m+1}\cup V) = (\bigcup_{i=1}^{k-m} (U_i\cap U_{k-m+1})) \cup V$ is a nest. Because of Theorem \ref{Nest}, we can apply Corollary \ref{GlueZL} to $X_k$ and $U_{k-m+1}\cup V$, whose union is $V_{k+1}$. Doing this for $k$ increasing from $m$ to $n+m-1$ proves the result.
\end{proof}

\begin{theorem} \label{H1basic}Suppose $U$ is an almost basic set. Then $H^1(U,\mathcal{LIP})=0.$
\end{theorem}

\begin{proof} From Lemma \ref{longexact} it is enough to show that every section $s\in \mathcal{ZL}(U)$ is in $\textrm{ZL}(U)=\mathcal{Z(U)}/\textrm{LIP}(U).$ By the definition of sheafification, $s$ comes from some open cover $\{U_i\}_{i=1}^{\infty}$ of $U$. We can assume that all $U_i$ are almost basic sets. Moreover, for $U=\{u_1,u_2 ,\dots\}$ we can refine this cover so that each $U_i$ contains $u_i$ and does not contain any $u_j$ for $j<i.$ 

We will construct an increasing sequence of nests $N_1\subseteq N_2\subseteq \dots$ that are made out of the sets from the given cover $\{U_i\}$  so that $\bigcup N_i=U$. This is done inductively as follows. For $n=1$ we take $N_1=U_1.$ Then for each $n$ we first consider all sets that $N_n$ is made out of, add the first unused $U_i$ to this collection, and consider this to be the straw part of the nest $N_{n+1}$. Then we add two more sets from the collection $\{U_i\}$ to create a nest as follows. We first consider all primes that divide any of the common differences of the open sets in our collection, except the primes that divide the common difference of $U$. We take an element $k_1$ in $U$ that is divisible by all of them and take $V_1$ to be $U_{k_1}.$ Note that $V_1$ contains $k_1$ and  intersects with all $U_i$ from our collection. Then we take $k_2$ in $U$ to be divisible by all primes dividing $k_1$ and the common difference of $U_{k_1},$ but not the common difference of $U.$ We take $V_2=U_{k_2}.$ The resulting collection of the sets, with $V=V_1\cup V_2$ is going to be a nest that we will call $N_{n+1}$. Indeed, the intersection properties of the nest follow from the construction. And the second condition follows from the fact that every prime that divides common differences of both $V_1$ and $V_2$ must divide the common difference of $U$.

From Theorem \ref{H1nest}, for every $k$ we can find LIP functions on the open sets $U_i$ that make the nest $N_k$ so that their differences restrict to the given elements on the intersections. When we go from $N_k$ to $N_{k+1},$ for a given $U_i$ that is already part of $N_k$, this function may need to be changed. However, we will prove that these changes can be made to not affect the values of the functions for all $x\leq k$ (note that $N_k$ contains all numbers from $1$ to $k$.) Indeed, the difference between the restriction to $N_k$ of the element in $\textrm{ZL}(N_{k+1})$ and the element in $\textrm{ZL}(N_k)$ is zero. So the difference of representatives is some LIP function $\mu_k\in \textrm{LIP}(N_k)$. Note that the system of functions for $N_{k+1}$ is defined up to a LIP function. So by subtracting a suitable integer polynomial, we can $\mu_k$ be zero on $\{1,2,\ldots, k\}.$

Putting all these systems of functions together for $k$ increasing to infinity, we get a collection of LIP functions $f_i\in \textrm{LIP}(U_i)$ with differences being the prescribed LIP functions on the intersection, proving that our element $s$ is in $\textrm{ZL}(U)$.
\end{proof}

The next result is a standard corollary; see, for example, the proof of Theorem 7.8.5 in \cite{Taylor}. We give an argument purely to keep the text more self-contained.
\begin{corollary}\label{H1Cech=H1} Suppose $U$ is open, and $C$ is its cover by almost basic sets. Then $H^1(U,\mathcal{LIP})\cong \check{H}^1_C(U,\mathcal{LIP}).$ 
\end{corollary}

\begin{proof} By Lemma \ref{longexact}, $H^1(U,\mathcal{LIP})$ is isomorphic to $\mathcal{ZL}(U) / \textrm{ZL}(U).$ Under this isomorphism, the natural homomorphism $\check{H}^1_C(U,\mathcal{LIP}) \to H^1(U,\mathcal{LIP})$ can be described as follows. For every element of $\check{H}^1_C(U,\mathcal{LIP})$ that is represented by $\{f_{ij} \in \textrm{LIP}(U_i\cap U_j)\}$  we can find $\{F_i: U_i\to \mathbb Z\}$ such that $f_{ij}=F_i-F_j.$ By gluing together the classes of $F_i$ in $\mathcal{Z}(U_i)/\textrm{LIP}(U_i)=\textrm{ZL}(U_i),$ we get a section of $\mathcal{ZL}(U).$ The difference of two systems of $\{F_i\}$ for the same $\{f_{ij}\}$ is a system of arbitrary $\mathbb Z$-valued functions on $U_i$ that match on all intersections, so they are restrictions of a function from $U$ to $\mathbb Z.$ Changing $\{f_{ij}\}$ to a different system representing the same class in $\check{H}^1_C(U,\mathcal{LIP})$ amounts to changing $\{F_i\}$ by adding to each of them a LIP function on $U_i.$ Thus, the map is well defined and, clearly, it is a homomorphism of abelian groups.

To show that this map is an isomorphism, we first prove that it is injective. Suppose we have a system $\{f_{ij}\}$ such that our construction gives zero. Since we have freedom in choosing $\{F_i\}$ up to the restrictions of a common $\mathbb Z$-valued function on $U,$ we can choose all $\{F_i\}$ to be zero. Thus, the original class is zero. 

The surjectivity is a bit more involved. Suppose we have an element in $\mathcal{ZL}(U)$. Then its restrictions to all $U_i$ are in $\mathcal{ZL}(U_i)$ that coincides with $\textrm{ZL}(U_i)$ by Theorem \ref{H1basic}. Note that for all $i$ and $j$ these restrictions restrict to the same on $U_i\cap U_j.$ We can lift them to $F_i: U_i \to \mathbb Z$; note that $F_i-F_j\in \textrm{LIP}(U_i\cap U_j).$ Calling these functions $f_{ij},$ we get a system of LIP functions that satisfy the cocycle condition, which gives an element in $\check{H}^1_C(U,\mathcal{LIP}) $ that we were looking for.
\end{proof}

\begin{corollary} Suppose $X$ is a nest. Then $H^1(X,\mathcal{LIP})=0.$
\end{corollary}

\begin{proof} This is a direct consequence of Theorem \ref{H1nest} and Corollary \ref{H1Cech=H1}.
\end{proof}

This applies, in particular, to star-like sets. It can also be generalized to sets considered  in Theorem \ref{Infstar} as follows.

\begin{theorem} \label{H1Infstar} Suppose $U$ is an open set and $a\in U.$ Denote by $U_a$ the union of all almost basic subsets of $U$ that contain $a.$ Then $H^1(U_a,\mathcal{LIP})=0$.
\end{theorem}

\begin{proof} Note that $U_a$ is a union of an increasing sequence of nests so we can proceed like in the end of the proof of Theorem \ref{H1basic}. 
\end{proof}

\begin{definition} We call an open set $X$ tree-like, if it can be covered by a finite or infinite collection of almost basic sets $U_i,$ such that the nerve of this covering is a tree. That is, the graph whose vertices are $U_i$ and two vertices are connected iff the sets have non-empty intersection, is a tree.
\end{definition}

\begin{theorem} Suppose $X$ is a tree-like set. Then $H^1(X,\mathcal{LIP})=0$.
\end{theorem}

\begin{proof} An infinite tree can be expressed as a union of finite trees. So, as in the proof of Theorem \ref{H1basic}, it is enough to prove it for $X$ that are covered by finite collection of almost basic sets, whose nerve is a tree. For such coverings $X=U_1\cup\dots \cup U_n$ we use induction on $n$. We can always find an end element of the graph, call it $U_n,$ so $U_n$ intersects with exactly one of the other $U_i.$ The induction hypothesis for $U_1\cup\dots \cup U_{n-1}$ and Theorems  \ref{H1basic} and \ref{MainH1} imply that $H^1(U_1\cup\dots \cup U_n, \mathcal{LIP})=0.$
\end{proof}

\begin{remark} The set $X$ from Theorem \ref{TnonLIP} is tree-like. So are the more general sets $\mathbb N \setminus (2p)\mathbb N$ for odd primes $p$. Note that not all locally LIP functions on such sets are LIP, however the first cohomology for the sheaf $\mathcal{LIP}$ is trivial.
\end{remark}

The following example shows that $H^1$ is not always zero.
\begin{theorem} \label{TH1not0} Suppose $U=\mathbb N \setminus 3\mathbb N$ and $W=\mathbb N \setminus 5\mathbb N.$ Then not every locally LIP function on $V=U\cap W$ is a sum of the restrictions of locally LIP functions on $U$ and $W$. Therefore, $H^1(U\cup W, \mathcal{LIP}) \neq 0.$
\end{theorem}

\begin{proof} Because $U$ is a disjoint union of two basic sets, $1+3\mathbb N_0$ and $2+3\mathbb N_0,$ $\mathcal{LIP} (U)$ consists of functions that are separately LIP on each of these basic sets. Similarly, every locally LIP function on $W$ is LIP on each of the four sets $a+5\mathbb N_0,$ where $a\in \{1,2,3,4\}.$ Locally LIP functions on $V$ are made of $8$ LIP functions on $a+15\mathbb N,$ where $a\in \{1,2,4,7,8,11,13,14\}.$

For each $f\in \mathcal{LIP}(V)$ consider $D_f=f(1)-f(7)-f(11)+f(17).$ For a restriction of a locally LIP function on $U$ this equals $(f(1)-f(7))-(f(11)-f(17)).$ Because $1$ and $7$ belong to a basic set, on which $f$ is LIP, $f(1)-f(7)$ is even. Likewise, $f(11)-f(17)$ is even, so the whole  $D_f$ is even.  Similarly, for a restriction of a locally LIP function on $W$ we rewrite $D_f$ as $(f(1)-f(11))-(f(7)-f(17)),$ so it is even. Consider now a function $f\in \mathcal{LIP}(V)$ such that $f(x)=1$ if $x\equiv 1 \pmod {15}$ and $f(x)=0$ otherwise. For this function $D_f=1,$ so $f$ is not a sum of restrictions of locally LIP functions on $U$ and $W.$
\end{proof}

We end this section with a theorem about triviality of higher $\check{\textrm{C}}$ech cohomology for star-like sets.

\begin{theorem} \label{THkstar} Suppose $U=U_1\cup \ldots \cup U_n$ is an almost star-like set, and $C$ is its cover consisting of the sets $U_i,$ $i=1,\dots, n.$ Then $\check{H}^k_C(U,\mathcal{LIP})=0$ for all $k\geq 1$.
\end{theorem}

\begin{proof} We essentially use a joint induction on $k$ and $n$. Specifically, suppose the statement is false, and $U=U_1\cup \ldots \cup U_n$ is a counterexample with the smallest $k$ and the smallest $n$ for that $k;$ we will show that it is not possible. From Theorem \ref{H1nest}, $k\geq 2.$ Also, $n\geq k+1$ by the definition of $\check{\textrm{C}}$ech cohomology.

Take a representative of an element of  $\check{H}^k_C(U,\mathcal{LIP}),$  which is a collection of LIP functions $f_{\Omega}$ on $U_{\Omega}=\bigcap \limits_{i\in \Omega} U_i$ for all subsets $\Omega$ of $\{1,\ldots, n\}$ of cardinality $k+1,$ which is exact.This means that for any set $I=\{i_1,\ldots,i_{k+2}\}$ with $1\leq i_1<\ldots <i_{k+2}\leq n$ the alternating sum $\sum \limits_{j=1}^n  (-1)^j  f_{I\setminus \{i_j\}}$ is $0$ on $U_{\Omega}$ (Here, again, we suppress in the notation the obvious restriction maps from $\textrm{LIP}(U_I)$ to $\textrm{LIP}(U_{I\setminus \{i_j\}}) $).

Consider only those $\Omega$ that do not contain $n.$ By the minimality of $n$, $\check{H}^k(U_1\cup \ldots \cup U_{n-1}=0,$ so by subtracting a suitable exact element from our representative, we can assume that $f_{\Omega}$ is nonzero only for $\Omega$ that contain $n.$ Now we consider a new set $U'=\bigcup \limits_{i=1}^{n-1} (U_i\cap U_n).$ The closed element that we have at the level $k+1$ for $U$ can be viewed as a closed element of level $k$ for $U'.$ By minimality of $k$, it is exact, which means that there is a collection of $f'_{\Omega '}$ for $\Omega '$ being subsets of cardinality $k-1$ of $\{1,\ldots,n-1\}$ such that our element is obtained from it by the usual coboundary map, i.e., taking an alternating sum of restrictions. This  collection of $f'_{\Omega '}$ can now be viewed as a collection at level $k$ for the original $U$. And if we augment it by zero functions for all $\Omega$ that do not contain $n,$ the coboundary map will send it to our representative, proving its exactness.
\end{proof}

\section{Conjectures and Comments}

Based on the available results and our intuition, we propose the following conjecture.

\begin{conjecture}
For all almost basic sets $U$ for all $k>0$ $H^k(U, \mathcal{LIP})=0.$  
\end{conjecture}

This conjecture would imply that $\check{C}$ech cohomology with respect to any cover by almost basic sets calculates the correct sheaf cohomology of $\mathcal{LIP}$ (cf. \cite{Taylor}, Theorem  7.8.5).  In particular, by Theorem \ref{THkstar} the vanishing of higher cohomology should generalize to almost star-like sets. One would also like to identify some natural category of  sheaves of modules over the $\mathcal{LIP}$ sheaf for which the same kind of vanishing would hold. 

Clearly, the obstruction in Theorem \ref{TH1not0} can be generalized significantly. It would be very interesting to figure out the exact structure of $H^1(U_{pq}, \mathcal{LIP})$ for all odd primes $p$ and $q$, where $U_{pq} = \mathbb N \setminus pq\mathbb N.$ There are reasons to believe that the Quadratic Reciprocity Law for $p$ and $q$ can be reproved or reinterpreted using this non-zero group or, perhaps, $H^1(U_{2pq}, \mathcal{LIP})$ or $H^2(U_{2pq}, \mathcal{LIP})$. Higher cohomology groups may be related to more complicated constructions, like triple symbols. 

In this paper we worked in the classical setting of Kirch topology on $\mathbb N.$ But the parallel constructions for $\mathbb Z \setminus \{0\}$ would work essentially the same way. Perhaps, the main thing that will be lost is the discrete derivative identity from Remark \ref{Rpowerseries}.  And in the proofs of some theorems we would choose an ordering on $\mathbb Z \setminus \{0\}$ to talk about ``first $n$ elements'' of its subsets. This clearly opens the door to generalizations to other rings, in particular to Dedekind domains (cf. \cite{Clarketal}).

Besides changing the input space, we can also change the target. In particular, it makes sense to look at localizations of $\mathbb Z$ (cf. \cite{MithunThesis}). Another natural target is integers in cyclotomic fields. Including them would allow to incorporate in our theory the general Dirichlet characters. 

As mentioned above, the first test of applicability for our theory may be to reprove or reinterpret the Quadratic Reciprocity Law. Higher reciprocity laws and Dirichlet Theorem on primes in arithmetic progressions are (perhaps, unreasonably ambitious) next goals. For applications like this, one should try to get some natural collection of LIP and locally LIP functions. One may also want to develop a version of our theory that incorporates the ``infinite prime'': include some conditions on the growth of LIP functions. These conditions would have to be sufficiently lenient to get a nontrivial theory (cf. \cite{LIP}).

Number-theoretic applications and meaning of Kirch topology have not been studied much. For example, Quadratic Reciprocity Law and Dirichlet's Theorem imply that the set of squares is closed in $\mathbb N.$ It is also connected and has a connected complement. This is a particular case of several natural phenomena (cf. \cite{BorisovSengupta}).

One would like to explore any possible connections with Arithmetic Topology, which is a much more sophisticated theory that claims analogy between $\mathbb Z$ and its finite extensions and the topology of three-dimensional manifolds (cf. \cite{Morishita}).

There must also exist strong connections with the p-adic Analysis. After all, product-sums converge with respect to the trivial norm on $\mathbb Z,$ so convergence with respect to the p-adic norm would be its close relative. Incidentally, for an element $a$ of the p-adic integers $\mathbb Z_p$ and an infinite subset $V$ of $\mathbb Z$ the product-sum for any LIP function $f$ on $V$ converges p-adically at $a$ if and only if there are infinitely many $v_i\in V$ such that $|v_i-a|_p <1.$ Moreover, in this case the sum does not depend on how we count elements of $V$ (cf. \cite{MithunThesis}). The fact that we do not need $a$ to be in the p-adic closure of $V$, i.e. congruences modulo higher power of $p$ do not come into play, seems intrinsically related to the great synergy between Kirch topology and LIP functions, to which this paper is a testament.

\vskip10pt\noindent {\bf Acknowledgments.} The author thanks Inna Sysoeva for many fruitful conversations regarding locally LIP functions in general and this paper in particular.

\end{document}